\documentclass [12pt]{article}
\usepackage{amssymb,amsmath,comment,amsthm}

\def\F{\mathbb{F}}

\newtheorem{theorem}{Theorem}[section]
\newtheorem{proposition}[theorem]{Proposition}
\newtheorem{corollary}[theorem]{Corollary}
\newtheorem{lemma}[theorem]{Lemma}

\newtheorem{conjecture}[theorem]{Conjecture}

\begin{document}
\title{A variant of Collatz's Conjecture over Binary Polynomials}
\author{ Luis H. Gallardo  and Olivier Rahavandrainy\\
Univ. Brest, UMR CNRS 6205\\
Laboratoire de Math\'ematiques de Bretagne Atlantique\\
6, Avenue Le Gorgeu, C.S. 93837, 29238 Brest Cedex 3, France.\\
e-mail: Luis.Gallardo@univ-brest.fr \\
Olivier.Rahavandrainy@univ-brest.fr}
\maketitle

\begin{itemize}
\item[a)]
Running head: Polynomial Collatz's Conjecture
\item[b)]
\item[c)]
\item[d)]
Corresponding author:
\begin{center} Luis H. Gallardo
\end{center}
\end{itemize}
\newpage~\\
{\bf{Abstract}}
We study a natural analogue of Collatz's Conjecture for polynomials over $\mathbb{F}_2$.
{\section{Introduction}}

The Collatz conjecture, also known as the Syracuse algorithm, is one of the most famous unsolved problems in number theory. Quoting \cite{ref-wiki}, ``it concerns sequences of integers in which each term is obtained from the previous term as follows: if the previous term is even, the next term is one half of the previous term. If the previous term is odd, the next term is $3$ times the previous term plus $1$. The conjecture states that these sequences always reach $1$, no matter which positive integer is chosen to start the sequence.''

It is natural to consider an analogue problem for polynomials over finite fields, which in some cases could be easier to treat. A number of attempts (probably not exhaustive) appear in \cite{Paran2023,Alon,Paran2025,Yucas,Matthews}. The first approach, due to Matthews \cite{Matthews}, can be loosely described as follows. Given fixed polynomials $K,D \in \mathbb{F}_q[x]$, we take a polynomial $A \in \mathbb{F}_q[x]$ and build a sequence with first term $A$, such that the next term $T$ corresponding to a given term $S$ is defined by
\begin{equation}\label{Matth}
T := \frac{K \cdot S - R}{D},
\end{equation}
where $K$ is coprime to $D$ and $R$ is a polynomial congruent to $Kr$ modulo $D$, with $r$ belonging to a complete set of residues modulo $D$. This guarantees that $T \in \mathbb{F}_q[x]$ in \eqref{Matth}.

When $q$ is even, many of the papers in the above literature study the case
$q=2$, $K = x$, and $R$ belonging to the set of residues $\{0,1\}$ modulo $x$. We may view the choice $K = x \in \mathbb{F}_2[x]$ as the analogue of the number $2 \in \mathbb{Z}$.

We believe it is also interesting to consider the case $q=2$, $K = x(x+1)$, since the rings $\mathbb{F}_2[x]$ and $\mathbb{F}_2[x+1]$ are indistinguishable. We have already considered (see \cite{Gall-Rahav2019}) an analogy between the ring of integers $\mathbb{Z}$ and the ring of binary polynomials $\mathbb{F}_2[x]$, in which the power $2^{a+b}$ in $\mathbb{Z}$ is replaced by the power $x^a(x+1)^b$ in $\mathbb{F}_2[x]$. This was used to study some ``Mersenne polynomials'' $M \in \mathbb{F}_2[x]$ defined by
\[ M := x^a(x+1)^b + 1, \]
the analogue of Mersenne numbers $2^{a+b} - 1$. These polynomials proved useful in solving some cases of the problem of finding fixed points of the sum-of-divisors function $\sigma \colon \mathbb{F}_2[x] \to \mathbb{F}_2[x]$ defined by
\begin{equation}\label{perfects}
\sigma(A) = \sum_{d \mid A} d,
\end{equation}
see \cite{Gall-Rahav2007, Gall-Rahav2016}.  

As shown by Matthews \cite{Matthews}, for many choices of $K,D,R$ in \eqref{Matth}, the sequence has divergent trajectories. For example, consider the transformation $T \colon \mathbb{F}_2[x] \to \mathbb{F}_2[x]$ defined by
\[
T(A) = \frac{A}{x}, \quad \text{if $x$ divides $A$,}
\]
and otherwise
\begin{equation}\label{MatthEx1}
T(A) = \frac{(x+1)^3 \cdot A + 1}{x}, \quad \text{if } A \equiv 1 \pmod{x}.
\end{equation}

Let us return to the original problem over the integers.
We may reformulate the Syracuse algorithm as follows: For a given positive integer $n$, consider the $2$-adic valuation $a_0$ of $n$: $n = 2^{a_0} n_1$, where $n_1$ is odd. Put $n_2 = 1 + 3n_1$, and again consider the $2$-adic valuation $a_2$ of $n_2$: $n_2 = 2^{a_2} n_3$ with $n_3$ odd, and so on.  

We obtain two sequences of odd and even integers: $[n_1, n_3,\ldots]$ and $[n_2, n_4,\ldots]$. The conjecture states that for any positive integer $n$, there exists an integer $m$ such that for all $t \geq m$, $n_{2t}= 2$ and $n_{2t+1}= 1$. Thus, the two sequences $(n_{2k})_k$ and $(n_{2k+1})_k$ are eventually constant.

Some papers (see, e.g., \cite{Furuta,Izadi}) attempt the difficult task of proving the original conjecture.

The following describes our approach, aimed at making progress in the binary polynomial case.

In the present paper we consider a variant of this problem for binary polynomials. Let $A \in \mathbb{F}_2[x]$ be a nonzero polynomial. We think of $x(x+1) \in \mathbb{F}_2[x]$ as a natural analogue of $2 \in \mathbb{Z}$. Following \cite{Gall-Rahav2007}, we say that $A \in \mathbb{F}_2[x]$ is \emph{odd} if $\gcd(A, x(x+1)) = 1$, i.e., if $A$ has no linear factor. Otherwise, $A$ is called \emph{even}. The first odd polynomial besides $1$ is $M_1 := x^2+x+1$. Thus, the analogue of $1+3n$ for integers is $1 + M_1 A$ for binary polynomials.

We denote by $val_x(S)$ (resp. $val_{x+1}(S)$) the valuation at $x$ (resp. at $x+1$) of a polynomial $S$. This means that
\[
S = x^{val_x(S)} (x+1)^{val_{x+1}(S)} S_1, \quad \text{where $S_1$ is odd}.
\]

For a fixed nonzero binary polynomial $A$, we define ``Collatz transformations'' obtained from $A$ by giving three sequences: an integer sequence $(a_{2k})_k$, and two polynomial sequences $(A_{2k})_k$, $(A_{2k+1})_k$ of even and odd polynomials respectively, associated with $A$, as follows:
\[
\begin{aligned}
&A_0 = A, \quad a_0 = val_x(A_0), \quad b_0 = val_{x+1}(A_0), \\
&A_1 \text{ is the odd polynomial such that } A_0 = x^{a_0}(x+1)^{b_0} A_1, \\
&A_2 = 1 + M_1 A_1, \quad a_2 = val_x(A_2), \ b_2 = val_{x+1}(A_2), \quad 
A_3 = \frac{A_2}{x^{a_2}(x+1)^{b_2}}, \\
&\vdots \\
&A_{2k} = 1+M_1 A_{2k-1}, \quad a_{2k}=val_x(A_{2k}), \ b_{2k}=val_{x+1}(A_{2k}), \quad 
A_{2k+1} = \frac{A_{2k}}{x^{a_{2k}}(x+1)^{b_{2k}}}, \\
&\vdots
\end{aligned}
\]

We may formulate a binary polynomial version of the Collatz conjecture as follows:

\begin{conjecture}\label{collatzconject}
For a given $A \in \mathbb{F}_2[x] \setminus \{0\}$, there exists $m \in \mathbb{N}^*$ such that for all $k \geq m$, $A_{2k} = x(x+1)$ and $A_{2k+1} = 1$.
\end{conjecture}

We prove this conjecture in the following theorem.

\begin{theorem}\label{collatz0}
Let $A$ be a nonzero binary polynomial. Then the sequences of polynomials obtained from Collatz transformations are of finite length $r_A$. More precisely, these sequences are
\[
[A_2, \ldots, A_{2m-2}, x^2+x] \quad \text{and} \quad [A_1, \ldots, A_{2m-1},1],
\]
where $m \in \mathbb{N}^*$ and $r_A = m+1 \leq 2^{\deg(A)-1}$.
\end{theorem}

The proof of Theorem \ref{collatz0} appears in Section \ref{theoremproof}.

The exponential upper bound for the length $r_A$ in the theorem seems too large, and we are unable to improve it.
In most cases in the literature \cite{Paran2023,Alon,Paran2025,Yucas}, the analogue of $r_A$ is proved to be bounded above by a polynomial in the degree of $A$. In these cases, 
$q=2$, $K = x$, and $R$ belongs to the set $\{0,1\}$ modulo $x$, while in our case 
$q=2$, $K = x(x+1)$, and $R$ belongs to $\{0,1,x,x+1\}$.

However, several computations suggest that here also $r_A$ is bounded by a polynomial in $\deg(A)$ (see Section \ref{exemples}). More precisely, Section \ref{exemples} contains details of the behavior of $r_A$ for sequences beginning with general polynomials $A$ of degree up to $35$. These computations motivate our conjecture as well as our comment on the upper bound of $r_A$. Moreover, in the same section we study $r_A$ in some special cases, including the case where $A$ is the trinomial $x^n+x+1$.

\section{Proof of Theorem \ref{collatz0}} \label{theoremproof}

First of all, we shall need some general results about the numbers of odd and even polynomials of a given degree $d \geq 2$.  
We let ${\mathcal{P}}_d$ denote the set of all polynomials of degree $d$ and consider:
\[
\begin{array}{l}
{\mathcal{P}}^{0,0}_d := \{S \in {\mathcal{P}}_d : S(0) = 0\}, \quad 
{\mathcal{P}}^{0,1}_d := \{S \in {\mathcal{P}}_d : S(0) = 1\}, \\[6pt]
{\mathcal{P}}^{1,0}_d := \{S \in {\mathcal{P}}_d : S(1) = 0\}, \quad 
{\mathcal{P}}^{1,1}_d := \{S \in {\mathcal{P}}_d : S(1) = 1\}, \\[6pt]
{\mathcal{O}}_d := \{S \in {\mathcal{P}}_d : S \text{ is odd}\} =  {\mathcal{P}}^{0,1}_d \cap {\mathcal{P}}^{1,1}_d.
\end{array}
\]

We have ${\mathcal{P}}_0 = {\mathcal{O}}_0 = \{1\}$ and ${\mathcal{O}}_1 = \emptyset$.

\begin{lemma} \label{oddpolyn}
\begin{itemize}
    \item[i)] The four sets ${\mathcal{P}}^{0,0}_d, {\mathcal{P}}^{0,1}_d, {\mathcal{P}}^{1,0}_d$, and ${\mathcal{P}}^{1,1}_d$ all have the same cardinality $2^{d-1}$.
    \item[ii)] The set ${\mathcal{O}}_d$ contains exactly $2^{d-2}$ polynomials if $d \geq 2$.
\end{itemize}
\end{lemma}

\begin{proof}
\begin{itemize}
    \item[i)] By the bijective map $S \mapsto S+1$, we have $\# {\mathcal{P}}^{0,0}_d = \# {\mathcal{P}}^{0,1}_d$ and $\# {\mathcal{P}}^{1,0}_d = \# {\mathcal{P}}^{1,1}_d$. Similarly, the bijection $S(x) \mapsto S(x+1)$ gives $\# {\mathcal{P}}^{0,0}_d = \# {\mathcal{P}}^{1,0}_d$. It remains to note that ${\mathcal{P}}_d$ is a disjoint union of ${\mathcal{P}}^{0,0}_d$ and ${\mathcal{P}}^{0,1}_d$, and that $\# {\mathcal{P}}_d = 2^d$.
    
    \item[ii)] By induction on $d$. The case $d = 2$ is trivial since ${\mathcal{O}}_2 = \{x^2 + x + 1\}$. Now, suppose that $\# {\mathcal{O}}_s = 2^{s-2}$ for $2 \leq s \leq d-1$. We observe that
    \[
    {\mathcal{P}}^{0,1}_d = \{(x+1)^s S_1 : 0 \leq s \leq d, \ s \neq d-1, \ S_1 \in {\mathcal{O}}_{d-s}\}.
    \]
    Hence,
    \[
    \# {\mathcal{P}}^{0,1}_d = \# {\mathcal{O}}_d + \# {\mathcal{O}}_{d-1} + \cdots + \# {\mathcal{O}}_2 + \# {\mathcal{O}}_0.
    \]
    Therefore, $2^{d-1} = \# {\mathcal{O}}_d + 2^{d-3} + \cdots + 1 + 1 = \# {\mathcal{O}}_d + 2^{d-2}$, and so $\# {\mathcal{O}}_d = 2^{d-2}$.
\end{itemize}
\end{proof}

For $k \geq 0$, we put $d_k := \deg(A_{2k})$ and $\ell_k := \deg(A_{2k+1})$.  
We then obtain the following lemmas.

\begin{lemma} \label{lesakbk}
We have $a_0, b_0 \geq 0$ and $a_{2k}, b_{2k} \geq 1$ for any $k \geq 1$.
\end{lemma}

\begin{proof}
If $k \geq 1$, then both $x$ and $x+1$ divide $A_{2k}$, because for $t \in \{0,1\}$,
\[
A_{2k}(t) = 1+M_1(t) A_{2k-1}(t) = 1 + 1 = 0.
\]
\end{proof}

\begin{lemma} \label{lesdkellk}
We have $\ell_{k+1} \leq \ell_k \leq \deg(A)$, $d_k = \ell_{k-1} + 2$, and $d_k = \ell_k + a_{2k} + b_{2k}$.
\end{lemma}

Since $(\ell_k)_k$ is a nonnegative and nonincreasing sequence, we obtain the following corollary.

\begin{corollary} \label{ellCV}
Both sequences $(\ell_k)_k$ and $(d_k)_k$ are convergent. We have
\[
\lim_{k \to \infty} \ell_k = \ell, \quad \lim_{k \to \infty} d_k = d, \quad \text{where } d = \ell + 2.
\]
\end{corollary}

\begin{corollary} \label{stationnaire}
There exists $m \geq 0$ such that for any $k \geq m$ one has
\[
\ell_k = \ell, \quad d_k = d, \quad a_{2k} = b_{2k} = 1.
\]
\end{corollary}

\begin{proof}
The convergent sequence $(\ell_k)_k$ takes its values in the finite set $\{0,1,\ldots, \deg(A)\}$. Hence it is eventually constant.
\end{proof}

\begin{corollary} \label{samedegrees}
For any $k \geq m$ we have $\deg(A_{2k}) = d$ and $\deg(A_{2k+1}) = \ell$.
\end{corollary}

\begin{corollary} \label{cyclic}
There exists a positive integer $t \leq \deg(A)$ such that the polynomials $A_{2(m+t)}$ and $A_{2(m+t)+1}$ satisfy
\[
A_{2(m+t)} = A_{2m}, \quad A_{2(m+t)+1} = A_{2m+1}.
\]
\end{corollary}

\begin{proof}
For any $k \geq m$, the polynomial $A_{2k}$ (resp. $A_{2k+1}$) lies in the finite set of polynomials of degree $d$ (resp. $\ell$).
\end{proof}

\begin{proposition} \label{lesAjodd=1}
For any $k \geq m$ we have $A_{2k+1} = 1$, so that $\ell = 0$ and $t = 1$.
\end{proposition}

\begin{proof}
For $k \geq m$ we have $a_{2k} = b_{2k} = 1$. Hence, the Collatz transformations give the following system:
\[
\left\{\begin{array}{l}
M_1 A_{2m+1} + (1+M_1) A_{2m+3} = 1,\\
M_1 A_{2m+3} + (1+M_1) A_{2m+5} = 1,\\
\quad \vdots \\
M_1 A_{2m+2t-3} + (1+M_1) A_{2m+2t-1} = 1,\\
M_1 A_{2m+2t-1} + (1+M_1) A_{2m+2t+1} = 1.
\end{array}\right.
\]
Since $A_{2m+2t+1} = A_{2m+1}$, we obtain a linear system of $t$ equations with coefficients in $\mathbb{F}_2[x]$ and
$t$ unknowns: $A_{2m+1}, \ldots, A_{2m+2t-1}$. Its matrix $C$ is circulant with first row $[M_1, 1+M_1, 0, \ldots, 0]$. The right-hand side is the transpose of $[1 \ldots 1]$.

By expanding along the first column of $C$, we find
\[
\det(C) = {M_1}^t + (1+M_1)^t,
\]
which is nonzero.

Thus, this system admits the unique solution $(1, \ldots, 1)$.
\end{proof}

\begin{corollary} \label{casell=0}
The even and odd sequences $E$ and $O$ are respectively
\[
E = [A_2, \ldots, A_{2m-2}, x^2+x], \quad O = [A_1, \ldots, A_{2m-1}, 1].
\]
Moreover, both $E$ and $O$ contain $m+1$ elements, with $m+1 \leq 2^{\deg(A)-1}$.
\end{corollary}

\begin{proof}
We have just seen that $\ell = 0$ and $t=1$. Thus, $d = 2$, $A_{2m+1} = 1$, and $A_{2m} = x^2+x$.  
The odd sequence $O$ contains at most all odd polynomials of degree $\deg(A)$, all odd polynomials of degree $\deg(A)-1, \ldots$, together with the polynomials $x^2+x+1$ and $1$.  

Thus, by Lemma \ref{oddpolyn}-ii), we obtain
\[
m+1 \leq 2^{\deg(A)-2} + 2^{\deg(A)-3} + \cdots + 2 + 1 + 1 = 2^{\deg(A)-1}.
\]
\end{proof}

\section{Some computer computations} \label{exemples}

After $7$ months and $3$ days of computation with {\tt GP-Pari}, we were able to determine, for every $n$ from $1$ to $35$, the maximal length $f(n)$ of the sequences $(A_j)$ that begin with some polynomial of degree $n$:

$$\begin{array}{|l|l|}
\hline
[n,f(n)]&[[1,0],[2,1],[3,2],[4,3],[5,4],[6,8],[7,10],[8,11],[9,12],[10,16]]\\
\hline
&[[11,18],[12,20],[13,22],[14,24],[15,28],[16,32],[17,36],[18,38]]\\
&[[19,40],[20,42],[21,46],[22,52],[23,54],[24,55],[25,60],[26,62]]\\
&[[27,66],[28,67],[29,70],[30,74],[31,76],[32,78],[33,84],[34,88]]\\
&[[35,92]]\\
\hline
\end{array}$$

From these computations, it appears plausible that $f(n)$ is bounded above by a polynomial in $n$, thus improving upon our exponential bound in Theorem \ref{collatz0}.

For each integer $n$ from $2$ to $38$, we also computed (with {\tt GP-Pari}, in about $5$ months and $8$ days) the number $g(n)$ of sequences $(A_j)$ that  
(a) begin with some odd polynomial $A_1=A$ of degree $n$, and  
(b) have the property that all elements $A_1,A_3,\ldots ,A_{2h+1}$ remain of degree $n$, and that the sequence is maximal with this property:

$$\begin{array}{|l|l|}
\hline
[n,g(n)]&[[2,1],[3,1],[4,2],[5,2],[6,3],[7,3],[8,4],[9,4],[10,5]]\\
\hline
&[[11,5],[12,6],[13,6],[14,7],[15,7],[16,8],[17,8],[18,9]]\\
&[[19,9],[20,10],[21,10],[22,11],[23,11],[24,12],[25,12],[26,13]]\\
&[[27,13],[28,14],[29,14],[30,15],[31,15],[32,16],[33,16],[34,17]]\\
&[[35,17],[36,18],[37,18],[38,19]]\\
\hline
\end{array}$$

From these values of $n$, one may conjecture that $g(n)= \lfloor n/2 \rfloor$.  
If the conjecture holds, we can upper bound the length $r_A$ of the sequence above by a polynomial in $n$ as follows:
\begin{equation}
\label{polyBound}
r_A \leq 2(n/2)+2((n-1)/2)+\cdots+2(1/2) = \tfrac{n(n+1)}{2}.
\end{equation} 

This improves upon the exponential upper bound of Theorem \ref{collatz0}. However, for example, when $n=21$, so that $n(n+1)/2=231$, an explicit polynomial $A$ with $g(21)=10$ has length $r_A=32$, while $f(21)=46$.

Concerning some special trinomials, after treating many cases by direct computation (with {\tt Maple}), we may state some examples as well as further conjectures in the next subsections.

\subsection{Example 1} \label{exemple1}
In this section, we take $A :=T_n= (x^2+x+1)^n +1 = {M_1}^n +1$, for $n \geq 1$, so that $A$ is even. Put $n=2^r u$, where $r \geq 0$ and $u$ is odd. We have
$$A = (M_1+1)^{2^r} \cdot (M^{u-1}+\cdots+M+1)^{2^r}.$$  
Hence, the first polynomial in the odd sequence is $A_1 = (M^{u-1}+\cdots+M+1)^{2^r}$.\\
On the other hand, if $n \geq 2$, then there exists a unique positive integer $r$ such that $2^{r-1} < n \leq 2^r$. Thus, we may write $n = 2^r - j$, with $0 \leq j \leq 2^{r-1}-1$. \\

Hence, we establish the following conjecture:

\begin{conjecture} \label{conjectpartcase1}
Let $A = {M_1}^{2^r-j} +1$ where $r \geq 1$ and $0 \leq j \leq 2^{r-1}-1$. Then, the length of the odd sequence of $A$ is $j+1$.
\end{conjecture}

For instance, below we show the odd degree sequences for $T_n = {M_1}^n+1$, with $n \in \{9,\ldots,16\}$ so that $n = 2^4 -j$, $0 \leq j \leq 7$. Here, all lengths are smaller than $n=\deg(T_n)/2$:

$$\begin{array}{|l|c|c|}
\hline
n&\text{Sequence associated with $T_n$}&\text{Length}\\
\hline
9&[16, 16, 16, 16, 16, 16, 16, 0]&8\\
10&[16, 16, 16, 16, 16, 16, 0]&7\\
11&[20, 16, 16, 16, 16, 0]&6\\
12&[16, 16, 16, 16, 0]&5\\
13&[24, 24, 24, 0]&4\\
14&[24, 24, 0]&3\\
15&[28, 0]&2\\
16&[0]&1\\
\hline
\end{array}$$

\subsection{Other small lengths} \label{othersmall}

We now give the odd degree sequences for $T_n = x^n+x+1$, $U_n = x^n+x^{n-1}+1$, and $S_n= x^n+x^7+x^3+1$, with $n \in \{31,32,33,34\}$. Here, all lengths are smaller than $\deg(A)/2+2$. Note that $S_n$ is even.

$$\begin{array}{|l|l|l|}
\hline
n&\text{Sequence associate with $T_n$}&\text{Length}\\
\hline
31&[31, 29, 24, 24, 16, 16, 16, 16, 0]&9\\
32&[32, 28, 24, 24, 16, 16, 16, 16, 0]&9\\
33&[33, 31, 30, 30, 28, 28, 28, 28, 24, 24, 24, 24, 24, 24, 24, 24, 0]&17\\
34&[34, 31, 30, 30, 28, 28, 28, 28, 24, 24, 24, 24, 24, 24, 24, 24, 0]&17\\
\hline
\end{array}$$

$$\begin{array}{|l|l|l|}
\hline
n&\text{Sequence associate with $U_n$}&\text{Length}\\
\hline
31&[31, 28, 24, 24, 16, 16, 16, 16, 0]&9\\
32&[32, 31, 30, 30, 28, 28, 28, 28, 24, 24, 24, 24, 24, 24, 24, 24, 0]&17\\
33&[33, 31, 30, 30, 28, 28, 28, 28, 24, 24, 24, 24, 24, 24, 24, 24, 0]&17\\
34&[34, 33, 30, 30, 28, 28, 28, 28, 24, 24, 24, 24, 24, 24, 24, 24, 0]&17\\
\hline
\end{array}$$

$$\begin{array}{|l|l|l|}
\hline
n&\text{Sequence associate with $S_n$}&\text{Length}\\
\hline
31&[30, 28, 28, 28, 28, 28, 24, 24, 0]&9\\
32&[28, 27, 24, 23, 16, 16, 16, 16, 0]&9\\
33&[32, 31, 22, 22, 16, 16, 16, 16, 0]&9\\
34&[32, 32, 30, 30, 27, 27, 27, 25, 24, 24, 24, 24, 24, 24, 24, 24, 0]&17\\
\hline
\end{array}$$
We have two conjectures for $A = x^n+x+1$ when $n \geq 4$:

\begin{conjecture} \label{proppartcase}~\\
Let $s$ be the greatest positive integer such that  $n-2^{s+1} \geq 0$. Then, for any positive integer $t \leq s-1$, the odd sequence contains $2^t$ polynomials which have the same degree $d_t$.
In particular, $d_1 = \deg(A_5) = \deg(A_7)$ and  $d_2 = \deg(A_9) = \deg(A_{11}) = \deg(A_{13})=\deg(A_{15})$.
\end{conjecture}

\begin{conjecture} \label{theorempartcase}
Let $s$ be the greatest positive integer such that  $n-2^{s+1} \geq 0$. Then,  both the even and odd sequences of $A$ have length $2^s+1$.
\end{conjecture}

Conjecture \ref{theorempartcase} follows from
Conjecture \ref{proppartcase}. Indeed, from Corollary \ref{casell=0}, the sequence of the odd polynomials is
$$[A_1, A_3, A_5, A_7, \ldots, A_{2m-1-2^s}, \ldots,A_{2m-3},A_{2m-1}, 1], \text{  of length $m+1$}.$$
One has, by Conjecture \ref{proppartcase}, $$m+1=1+1+2+2^2+\cdots +2^{s-1} + 1 = 1+(2^s-1) + 1 = 2^s+1.$$

\subsection{More small lengths}
Here, the lengths are all equal to $\deg(A)+1$. Define
\[
P_1=x^{14} + x^{10} + x^9 + x^8 + x^3 + x^2 + 1, \quad 
P_2 = x^{14} + x^{12} + x^9 + x^6 + x^4 + x^3 + 1,
\]
and
\[
P_3 = x^{14} + x^{13} + x^9 + x^8 + x^6 + x^5 + x^3 + x^2 + 1.
\]
Observe that $P_1$ and $P_2$ have the same odd degree sequences.

$$\begin{array}{|l|l|l|}
\hline
\text{Polynomial}&\text{Sequence}&\text{Length}\\
\hline
P_1&[14, 14, 13, 13, 12, 12, 10, 10, 9, 8, 7, 5, 4, 3, 0]&15\\
P_2&[14, 14, 13, 13, 12, 12, 10, 10, 9, 8, 7, 5, 4, 3, 0]&15\\
P_3&[14, 14, 13, 13, 11, 11, 10, 9, 8, 7, 5, 5, 4, 3, 0]&15\\
\hline
\end{array}$$

\subsection{Some technical remarks} \label{remarques}
We let $\overline{S}$ denote the polynomial obtained from $S \in \F_2[x]$ by replacing $x$ with $x+1$. We also consider the reciprocal $S^*$ of $S$, defined as
\[
S^*(x) = x^{\deg(S)} \cdot S \!\left(\tfrac{1}{x}\right).
\]

It is easy to see that the Collatz sequences of $\overline{A}$ are exactly obtained from those of $A$ by applying the operation $S \mapsto \overline{S}$.  
However, for $A^*$, this is not generally true. For example, if $A=x^8+x^3+1$, then the odd sequence is $[8, 7, 5, 5, 4, 3, 0]$, whereas for $A^*=x^8+x^5+1$, one obtains $[8, 6, 6, 0]$.  
Nonetheless, if $A$ is of the form $x^n+x+1$, then the odd degree sequence associated with $A^*$ is different (in general), but still regular enough (see Section \ref{othersmall}, with $U_{32} = {T_{32}}^*$).

\end{document}